\documentclass{amsart}%[a4paper]
\usepackage{graphicx}
\usepackage{amsmath}
\usepackage{amssymb}
\usepackage[all]{xy}

\newcommand{\Hom}{\operatorname{Hom}\nolimits}
\renewcommand{\Im}{\operatorname{Im}\nolimits}

\newcommand{\Coker}{\operatorname{Coker}\nolimits}

\newcommand{\depth}{\operatorname{depth}\nolimits}

\newcommand{\Tor}{\operatorname{Tor}\nolimits}
\newcommand{\Tatetor}{\operatorname{\widehat{Tor}}\nolimits}
\newcommand{\Ext}{\operatorname{Ext}\nolimits}

\renewcommand{\H}{\operatorname{H}\nolimits}

\newcommand{\m}{\operatorname{\mathfrak{m}}\nolimits}

\newcommand{\cx}{\operatorname{cx}\nolimits}

\newcommand{\reddeg}{\operatorname{reddeg}\nolimits}

\newcommand{\s}{\operatorname{\Sigma}\nolimits}

\newtheorem{theorem}{Theorem}[section]

\newtheorem{corollary}[theorem]{Corollary}

\newtheorem{proposition}[theorem]{Proposition}
\theoremstyle{definition}
\newtheorem*{definition}{Definition}
\theoremstyle{definition}

\newtheorem{construction}[theorem]{Construction}
\theoremstyle{definition}

\theoremstyle{definition}

\theoremstyle{definition}

\theoremstyle{definition}

\theoremstyle{remark}
\newtheorem*{remark}{Remark}
\theoremstyle{definition}

\theoremstyle{definition}

\begin{document}

\title{The depth formula for modules with reducible complexity}
\author{Petter Andreas Bergh \& David A.\ Jorgensen}

\address{Petter Andreas Bergh \\ Institutt for matematiske fag \\
  NTNU \\ N-7491 Trondheim \\ Norway}
\email{bergh@math.ntnu.no}

\address{David A.\ Jorgensen \\ Department of mathematics \\ University
of Texas at Arlington \\ Arlington \\ TX 76019 \\ USA}
\email{djorgens@uta.edu}

%\date{\today}

\begin{abstract} We prove that the depth formula holds for $\Tor$-independent modules
in certain cases over a Cohen-Macaulay local ring, provided one of the modules has
reducible complexity.
\end{abstract}



\maketitle

\section{Introduction}\label{secintro}

Two finitely generated modules $M$ and $N$ over a local ring $A$
satisfy the \emph{depth formula} if
\begin{equation}\label{depthform}
\depth M + \depth N = \depth A + \depth (M \otimes_A N)
\end{equation}
This formula is not at all true in general; an obvious
counterexample appears by taking modules of depth zero over a ring
of positive depth. The natural question is then: for which pairs of
modules does the formula hold? The first systematic treatment of this
question was done by Auslander in \cite{Auslander}, where he
considered the case when one of the modules involved has finite
projective dimension. In this situation, let $q$ be the largest
integer such that $\Tor^A_q(M,N)$ is nonzero. Auslander proved that
if either $\depth \Tor^A_q(M,N) \le 1$ or $q=0$, then the formula
\begin{equation}\label{ausform}
\depth M + \depth N = \depth A + \depth \Tor^A_q(M,N) -q
\end{equation}
holds. The case $q=0$ is the depth formula.

Auslander's result indicated that in order to decide which pairs of
modules satisfy the depth formula, one should concentrate on
\emph{$\Tor$-independent} pairs, that is, modules $M$ and $N$
satisfying $\Tor^A_n(M,N)=0$ for $n>0$. In \cite{HunekeWiegand1},
Huneke and Wiegand showed that the depth formula holds for such
modules over complete intersections. This (and Auslander's result)
was later generalized in \cite{ArayaYoshino} by Araya and Yoshino,
who considered the case when one of the modules involved has finite
complete intersection dimension. If then $\Tor^A_n(M,N)=0$ for $n
\gg 0$, let $q$ be the largest integer such that $\Tor^A_q(M,N)$ is
nonzero. In this situation, Araya and Yoshino proved Auslander's
original result (\ref{ausform}) above if either
$\depth \Tor^A_q(M,N) \le 1$ or $q=0$.

The aim of this paper is to investigate the depth formula
(\ref{depthform}) for $\Tor$-independent modules over a local
Cohen-Macaulay ring, provided one of the modules has reducible
complexity. In particular, we show that the formula (\ref{depthform})
holds for $\Tor$-independent modules over a Cohen-Macaulay
ring if one module has reducible complexity and is not maximal
Cohen-Macaulay, or if  \emph{both} modules have reducible
complexity. Moreover, we prove that the depth formula holds if
one of the modules involved has reducible complexity, and the
other has finite Gorenstein dimension. 

In the final section we show that
there exist modules having reducible complexity of any finite complexity,
but not finite complete intersection dimension. Knowing that
such modules exist is a critical point of the investigation. Modules of infinite complete intersection dimension are in a precise
sense far from resembling those modules considered in the original explorations of the formulas (\ref{depthform}) and (\ref{ausform}).  
Thus we show that the depth formula holds in a context that is fundamentally departed from previous considerations. We know of no example of finitely generated Tor-independent modules that do not satisfy the depth formula (\ref{depthform}), nor are we aware of a counterexample to Auslander's
formula (\ref{ausform}) when $q<\infty$, and $\depth\Tor^R_q(M,N)\le 1$ or $q=0$.

\section{Reducible complexity}\label{secredcx}

Throughout the rest of this paper, we assume that all modules
encountered are finitely generated. In this section, we fix a
local (meaning commutative Noetherian local) ring $(A, \m, k)$.
Under these assumptions, every $A$-module $M$ admits a minimal
free resolution
$$\cdots \to F_2 \to F_1 \to F_0 \to M \to 0$$
which is unique up to isomorphism. The rank of the free $A$-module
$F_n$ is the $n$th \emph{Betti number} of $M$; we denote it by
$\beta_n^A(M)$. The \emph{complexity} of $M$, denoted $\cx M$, is
defined as
$$\cx M \stackrel{\text{def}}{=} \inf \{ t \in \mathbb{N} \cup
\{ 0 \} \mid \exists a \in \mathbb{R} \text{ such that } \beta_n^A
(M) \le an^{t-1} \text{ for all } n \gg 0 \}.$$ In other words, the
complexity of a module is the polynomial rate of growth of its Betti
sequence. It follows from the definition that $\cx M=0$ precisely
when $M$ has finite projective dimension, and that $\cx M=1$ if and
only if the Betti sequence of $M$ is bounded. An arbitrary local
ring may have many modules with infinite complexity; by a theorem of
Gulliksen (cf.\ \cite{Gulliksen}), the local rings over which all
modules have finite complexity are precisely the complete
intersections.

In \cite{Bergh1}, the concept of modules with reducible complexity
was introduced. These are modules which in some sense generalize
modules of finite complete intersection dimension (see
\cite{AvramovGasharovPeeva}), in particular modules over complete
intersections. Before we state the definition, we recall the
following. Let $M$ and $N$ be $A$-modules, and consider an element
$\eta \in \Ext^t_A(M,N)$. By choosing a map $f_{\eta} \colon
\Omega_A^t(M) \to N$ representing $\eta$, we obtain a commutative
pushout diagram
$$\xymatrix{
0 \ar[r] & \Omega_A^t(M) \ar[r] \ar[d]^{f_{\eta}} & F_{t-1} \ar[r]
\ar[d] & \Omega_A^{t-1}(M) \ar[r]
\ar@{=}[d] & 0 \\
0 \ar[r] & N \ar[r] & K_{\eta} \ar[r] & \Omega_A^{t-1}(M) \ar[r] & 0
}$$ with exact rows. The module $K_{\eta}$ is independent, up to
isomorphism, of the map $f_{\eta}$ chosen as a representative for
$\eta$. We now recall the definition of modules with reducible
complexity. Given $A$-modules $X$ and $Y$, we denote the graded
$A$-module $\oplus_{i=0}^{\infty} \Ext_A^i(X,Y)$ by $\Ext_A^*(X,Y)$.
\begin{definition}\label{defredcx}
The full subcategory of $A$-modules consisting of the modules
having \emph{reducible complexity} is defined inductively as
follows:
\begin{enumerate}
\item[(i)] Every $A$-module of finite projective dimension has reducible
complexity.
\item[(ii)] An $A$-module $M$ of finite positive complexity
has reducible complexity if there exists a homogeneous element $\eta
\in \Ext_A^* (M,M)$, of positive degree, such that $\cx K_{\eta} <
\cx M$ and $K_{\eta}$ has reducible complexity.
\end{enumerate}
\end{definition}
Thus, an $A$-module $M$ of finite positive complexity $c$, say, has
reducible complexity if and only if the following hold: there exist
nonnegative integers $n_1, \dots, n_t$, with $t \le c$, and exact
sequences (with $K_0=M$)
$$\xymatrixrowsep{.6pc} \xymatrix{
\eta_1 \colon &  0 \ar[r] & K_0 \ar[r] & K_1 \ar[r] & \Omega_A^{n_1}(K_0) \ar[r] & 0
\\
& & \vdots & \vdots & \vdots \\
\eta_t \colon & 0 \ar[r] & K_{t-1} \ar[r] & K_t \ar[r] &
\Omega_A^{n_t}(K_{t-1}) \ar[r] & 0 }$$ in which $\cx M > \cx K_1 >
\cdots > \cx K_t =0$. We say that these sequences $\eta_1, \dots,
\eta_t$ \emph{reduce the complexity} of $M$. As shown in
\cite{Bergh1}, every module of finite complete intersection
dimension has reducible complexity. In particular, if $A$ is a
complete intersection, then every $A$-module has this property.

In the original definition in \cite{Bergh1}, the extra requirement
$\depth M = \depth K_1 = \cdots = \depth K_t$ was included. However,
as we will only be working over Cohen-Macaulay rings, this
requirement is redundant. Namely, when $A$ is Cohen-Macaulay and $M$
is any $A$-module, then the depth of any syzygy of $M$ is at least
the depth of $M$. Consequently, in a short exact sequence
$$0 \to M \to K \to \Omega_A^n(M) \to 0$$
the depth of $M$ automatically equals that of $K$.

\section{The depth formula}\label{secdepthformula}

Let $A$ be a local ring, and let $M$ be an $A$-module with
reducible complexity. If the complexity of $M$ is positive, then
by definition there exist a number $t$ and short exact sequences
$$\eta_i: \hspace{2mm} 0 \to K_{i-1} \to K_i \to
\Omega_A^{|\eta_i|-1}(K_{i-1}) \to 0$$ for $1 \le i \le t$ reducing
the complexity of $M$. We define the \emph{upper reducing degree} of
$M$, denoted $\reddeg^* M$, to be the supremum of the minimal degree
of the cohomological elements $\eta_i$, the supremum taken over all
such sequences reducing the complexity of $M$:
$$\reddeg^* M \stackrel{\text{def}}{=} \sup \{ \min \{ |\eta_1|,
\dots, |\eta_t| \} \mid \eta_1, \dots, \eta_t \text{ reduces the
complexity of } M \}.$$ If the complexity of $M$ is zero, that is,
if $M$ has finite projective dimension, then we define $\reddeg^* M
= \infty$. Note that the inequality $\reddeg^* M \ge 1$ always
holds.

We now prove our first result, namely the depth formula in the situation
when the tensor product of
the two modules involved has depth zero. In this result, we also
include a generalized version of half of \cite[Theorem
2.5]{ArayaYoshino}.

\begin{theorem}\label{firstdepthresult}
Let $A$ be a Cohen-Macaulay local ring, and let $M$ and $N$ be
nonzero $A$-modules such that $M$ has reducible complexity. Suppose
that $\Tor_n^A(M,N)=0$ for $n \gg 0$, and let $q$ be the largest
integer such that $\Tor^A_q(M,N)$ is nonzero. Furthermore, suppose
that one of the following holds:
\begin{enumerate}
\item[(i)] $\depth \Tor^A_q(M,N)=0$,
\item[(ii)] $q \ge 1$, $\depth \Tor^A_q(M,N) \le 1$ and $\reddeg^* M \ge
2$.
\end{enumerate}
Then the formula
$$\depth M + \depth N = \dim A + \depth \Tor^A_q(M,N) -q$$
holds.
\end{theorem}

\begin{proof}
Part (i) is just \cite[Theorem 3.4(i)]{Bergh1}, so we only need to
prove (ii). We do this by induction on the complexity of $M$, where
the case $\cx M=0$ follows from Auslander's original result
\cite[Theorem 1.2]{Auslander}. Suppose therefore the complexity of
$M$ is nonzero. Since $\reddeg^* M \ge 2$, there exists an exact
sequence
$$0 \to M \to K \to \Omega_A^n(M) \to 0$$
with $n \ge 1$, in which the complexity of $K$ is at most $\cx M -1$
and $\reddeg^* K \ge 2$. From this sequence we see that
$\Tor^A_q(K,N)$ is isomorphic to $\Tor^A_q(M,N)$, and that
$\Tor^A_i(K,N)=0$ for $i \ge q+1$. The formula therefore holds with
$K$ replacing $M$, but since $\depth K = \depth M$ we are done.
\end{proof}

As mentioned, this result generalizes the first half of
\cite[Theorem 2.5]{ArayaYoshino}. Namely, if $A$ is a local ring and
$M$ is a module of finite complete intersection dimension, then $M$
has reducible complexity by \cite[Proposition 2.2(i)]{Bergh1}, and
$\reddeg^* M = \infty$ by \cite[Lemma 2.1(ii)]{Bergh2}.

Next, we show that the depth formula is valid for $\Tor$-independent
modules over a local Cohen-Macaulay ring in the following situation:
one of the modules has reducible complexity, and the other has
finite Gorenstein dimension. Recall therefore that if $A$ is a local
ring, then a module $M$ has \emph{Gorenstein dimension zero} if the
following hold: the module is reflexive (i.e.\ the canonical
homomorphism $M \to \Hom_A ( \Hom_A(M,A),A)$ is bijective), and
$$\Ext_A^n(M,A)=0= \Ext_A^n ( \Hom_A(M,A),A)$$
for all $n>0$. The \emph{Gorenstein dimension} of $M$ is defined to
be the infimum of all nonnegative integers $n$, such that there
exists an exact sequence
$$0 \to G_n \to \cdots \to G_0 \to M \to 0$$
in which all the $G_i$ have Gorenstein dimension zero. By
\cite[Theorem 4.13]{AuslanderBridger}, if $M$ has finite Gorenstein
dimension, then it equals $\depth A - \depth M$. Moreover, by
\cite[Theorem 4.20]{AuslanderBridger}, a local ring is Gorenstein
precisely when \emph{every} module has finite Gorenstein dimension.

\begin{proposition}\label{nonzerodepth}
Let $A$ be a local Cohen-Macaulay ring, and $M$ and $N$ be nonzero
$\Tor$-independent $A$-modules. Assume that $M$ is maximal
Cohen-Macaulay and has reducible complexity, and that $N$ has finite
Gorenstein dimension. Then if $\depth (M \otimes_A N)$ is nonzero,
so is $\depth N$.
\end{proposition}

\begin{proof}
By \cite[Lemma 2.17]{ChristensenFrankildHolm}, there exists an exact
sequence
$$0 \to N \to I \to X \to 0$$
in which the projective dimension of $I$ is finite and $X$ has
Gorenstein dimension zero. Then $\Tor^A_n(M,I)$ and $\Tor^A_n(M,X)$
both vanish for $n \gg 0$, and since $M$ is maximal Cohen-Macaulay
it follows from \cite[Theorem 3.3]{Bergh1} that $\Tor^A_n(M,I)=0=
\Tor^A_n(M,X)$ for $n \ge 1$. Hence the pairs $(M,N), (M,I)$ and
$(M,X)$ are all $\Tor$-independent.

Suppose $\depth N=0$. Then the depth of $I$ is also zero. Tensoring
the exact sequence with $M$ yields the exact sequence
$$0 \to M \otimes_A N \to M \otimes_A I \to M \otimes_A X \to 0.$$
By Auslander's original result, the depth formula holds for the
pair $(M,I)$. Moreover, by \cite[Theorem 3.4(iii)]{Bergh1}, the
formula also holds for the pair $(M,X)$, hence
$$\depth M + \depth I = \dim A + \depth (M \otimes_A I)$$
and
$$\depth M + \depth X = \dim A + \depth (M \otimes_A X).$$
The first of these formulas implies that the depth of $M \otimes_A
I$ is zero. The second formula, together with the fact that $X$ is
maximal Cohen-Macaulay, implies that $M \otimes_A X$ is maximal
Cohen-Macaulay. Therefore $\depth (M \otimes_A N)=0$ by the depth
lemma.
\end{proof}

We can now prove that the depth formula holds when one module has
reducible complexity, and the other has finite Gorenstein dimension.

\begin{theorem}[Depth formula - Gorenstein case 1]\label{GdimFinite}
Let $A$ be a local Cohen-Macaulay ring, and $M$ and $N$ be nonzero
$\Tor$-independent $A$-modules. If $M$ has reducible complexity and
$N$ has finite Gorenstein dimension, then
$$\depth M + \depth N = \dim A + \depth (M \otimes_A N).$$
\end{theorem}

\begin{proof}
We prove this result by induction on the depth of the tensor
product. If $\depth (M \otimes_A N)=0$, then the formula holds by
Theorem \ref{firstdepthresult}, so assume that $\depth (M \otimes_A
N)$ is positive. If $M$ has finite projective dimension, then the
formula holds by Auslander's original result, hence we assume that
the complexity of $M$ is positive.

Suppose the depth of $N$ is zero. Choose short exact sequences (with
$K_0=M$)
$$\xymatrixrowsep{.6pc} \xymatrix{
0 \ar[r] & K_0 \ar[r] & K_1 \ar[r] & \Omega_A^{n_1}(K_0) \ar[r] & 0
\\
& \vdots & \vdots & \vdots \\
0 \ar[r] & K_{t-1} \ar[r] & K_t \ar[r] & \Omega_A^{n_t}(K_{t-1})
\ar[r] & 0 }$$ reducing the complexity of $M$, and note that the
pair $(K_i,N)$ is $\Tor$-independent for all $i$. Since the
projective dimension of $K_t$ is finite, the depth formula holds for
$K_t$ and $N$, i.e.\
$$\depth K_t + \depth N = \dim A + \depth (K_t \otimes_A N).$$
Since $\depth N=0$, we see that $K_t$, and hence also $M$, is
maximal Cohen-Macaulay. But this contradicts Proposition
\ref{nonzerodepth}, hence the depth of $N$ must be positive.

Choose an element $x \in A$ which is regular on both $N$ and $M
\otimes_A N$. Tensoring the exact sequence
$$0 \to N \xrightarrow{\cdot x} N \to N/xN \to 0$$
with $M$, we get the exact sequence
$$0 \to \Tor^A_1(M,N/xN) \to M \otimes_A N \xrightarrow{\cdot x} M \otimes_A
N \to M \otimes_A N/xN \to 0.$$ We also see that
$\Tor_n^A(M,N/xN)=0$ for $n \ge 2$. However, the element $x$ is
regular on $M \otimes_A N$, hence $\Tor^A_1(M,N/xN)=0$ and $(M
\otimes_A N)/x(M \otimes_A N) \simeq M \otimes_A N/xN$. The modules
$M$ and $N/xN$ are therefore $\Tor$-independent, and $\depth (M
\otimes_A N/xN) = \depth (M \otimes_A N)-1$. By induction, the depth
formula holds for $M$ and $N/xN$, giving
\begin{eqnarray*}
\depth M + \depth N & = & \depth M + \depth N/xN +1 \\
& = & \dim A + \depth (M \otimes_A N/xN) +1 \\
& = & \dim A + \depth (M \otimes_A N).
\end{eqnarray*}
This concludes the proof.
\end{proof}

\begin{corollary}[Depth formula - Gorenstein case 2]\label{Gorensteincase}
Let $A$ be a Gorenstein local ring, and $M$ and $N$ be nonzero
$\Tor$-independent $A$-modules. If $M$ has reducible complexity,
then
$$\depth M + \depth N = \dim A + \depth (M \otimes_A N).$$
\end{corollary}

\begin{remark}
In work in progress by Lars Winther Christensen and the second
author (cf.\ \cite{ChristensenJorgensen}), the depth formula is
proved for modules $M$ and $N$ over a local ring $A$ under the
following assumptions: the module $M$ has finite Gorenstein
dimension, and the Tate homology group $\Tatetor^A_n(M,N)$ vanishes
for all $n \in \mathbb{Z}$.
\end{remark}

What can we say if the ring is not necessarily Gorenstein, or, more
general, when we do not assume that one of the modules has finite
Gorenstein dimension? The following result shows that if the ring is
Cohen-Macaulay and the module having reducible complexity is not
maximal Cohen-Macaulay, then the depth formula holds.

\begin{theorem}[Depth formula - Cohen-Macaulay
case 1]\label{CMcaseone} Let $A$ be a Cohen-Macaulay local ring,
and let $M$ and $N$ be nonzero $\Tor$-independent $A$-modules. If
$M$ has reducible complexity and is not maximal Cohen-Macaulay,
then
$$\depth M + \depth N = \dim A + \depth (M \otimes_A N).$$
\end{theorem}

\begin{proof}
We prove this result by induction on the complexity of $M$. As
before, if $M$ has finite projective dimension, then the depth
formula follows from Auslander's original result. We therefore
assume that the complexity of $M$ is positive.

Choose a short exact sequence
$$0 \to M \to K \to \Omega_A^t(M) \to 0$$
in $\Ext_A^t(M,M)$, with $\cx K < \cx M$ and $t \ge 0$. Since $M$
and $K$ are $\Tor$-independent and $\depth K = \depth M$, the
depth formula holds for these modules by induction, i.e.\
\begin{equation*}\label{ES1}
\depth K + \depth N = \dim A + \depth (K \otimes_A N).
\tag{$\dagger$}
\end{equation*}
Therefore, we need only to show that $\depth (K \otimes_A N) =
\depth (M \otimes_A N)$.

If $t=0$, then by tensoring the above exact sequence with $N$, we
obtain the exact sequence
$$0 \to M \otimes_A N \to K \otimes_A N \to M \otimes_A N \to 0.$$
In this situation, the equality $\depth (K \otimes_A N) = \depth
(M \otimes_A N)$ follows from the depth lemma, and we are done.
What remains is therefore the case $t \ge 1$. Moreover, by
considering the short exact sequence
\begin{equation*}\label{ES2}
0 \to M \otimes_A N \to K \otimes_A N \to \Omega_A^t(M) \otimes_A
N \to 0, \tag{$\dagger \dagger$}
\end{equation*}
we see that if the depth of $M \otimes_A N$ is zero, then so is
the depth of $K \otimes_A N$. In this case we are done, hence we
may assume that the depth of $M\otimes_A N$ is positive.

Suppose $\depth (K \otimes_A N) > \depth (M \otimes_A N)$. Then
$\depth (\Omega_A^t(M) \otimes_A N) = \depth (M \otimes_A N)-1$ by
the depth lemma. Now for each $i \ge 1$, let
$$0 \to \Omega_A^{i+1}(M) \to A^{\beta_i} \to \Omega_A^{i}(M) \to 0$$
be a projective cover of $\Omega_A^{i}(M)$, and note that this
sequence stays exact when we tensor with $N$. Let $s$ be the
largest integer in $\{ 0, \dots, t-1 \}$ such that in the exact
sequence
\begin{equation*}\label{ES3}
0 \to \Omega_A^{s+1}(M) \otimes_A N \to N^{\beta_s} \to
\Omega_A^{s}(M)\otimes_A N \to 0 \tag{$\dagger \dagger \dagger$}
\end{equation*}
the inequality $\depth (\Omega_A^{s+1}(M) \otimes_A N) < \depth
(\Omega_A^s(M) \otimes_A N)$ holds. From the depth lemma applied
to this sequence, we see that
\begin{eqnarray*}
\depth N & = & \depth (\Omega_A^{s+1}(M) \otimes_A N) \\
& \le & \depth (\Omega_A^t(M) \otimes_A N) \\
& = & \depth (M \otimes_A N)-1 \\
& < & \depth (K \otimes_A N) -1.
\end{eqnarray*}
But then from (\ref{ES1}) we obtain the contradiction $\dim A <
\depth K-1$, and consequently the inequality $\depth (K \otimes_A
N) > \depth (M \otimes_A N)$ cannot hold.

Next, suppose that $\depth (K \otimes_A N) < \depth (M \otimes_A
N)$. Applying the depth lemma to (\ref{ES2}), we see that $\depth
(K \otimes_A N) = \depth (\Omega_A^t(M) \otimes_A N)$. Again, let
$s$ be the largest integer in $\{ 0, \dots, t-1 \}$ such that
$\depth (\Omega_A^{s+1}(M) \otimes_A N) < \depth (\Omega_A^s(M)
\otimes_A N)$. Then the depth lemma applied to (\ref{ES3}) gives
\begin{eqnarray*}
\depth N & = & \depth (\Omega_A^{s+1}(M) \otimes_A N) \\
& \le & \depth (\Omega_A^t(M) \otimes_A N) \\
& = & \depth (K \otimes_A N) .
\end{eqnarray*}
From (\ref{ES1}) it now follows that $K$, and hence also $M$, is
maximal Cohen-Macaulay, a contradiction. This shows that the depth
of $K \otimes_A N$ equals that of $M \otimes_A N$.
\end{proof}

Next, we show that if both the $\Tor$-independent modules have
reducible complexity, then the depth formula holds without the
assumption that $M$ is not maximal Cohen-Macaulay.

\begin{theorem}[Depth formula - Cohen-Macaulay
case 2]\label{CMcasetwo} Let $A$ be a Cohen-Macaulay local ring,
and let $M$ and $N$ be nonzero $\Tor$-independent $A$-modules. If
both $M$ and $N$ have reducible complexity, then
$$\depth M + \depth N = \dim A + \depth (M \otimes_A N).$$
\end{theorem}

\begin{proof}
If one of the modules is not maximal Cohen-Macaulay, the result
follows from Theorem \ref{CMcaseone}. If not, then the result
follows from \cite[Theorem 3.4(iii)]{Bergh1}.
\end{proof}

What happens over a Cohen-Macaulay ring if we only require that
one of the modules has reducible complexity? We end this section
with the following result, showing that, in this situation, if the
depth of the tensor product is nonzero, then so is the depth of
the module having reducible complexity.

\begin{proposition}\label{positivedepth}
Let $A$ be a Cohen-Macaulay local ring, and let $M$ and $N$ be
nonzero $\Tor$-independent $A$-modules such that $M$ has reducible
complexity. Then if $\depth (M \otimes_A N)$ is nonzero, so is
$\depth M$.
\end{proposition}

\begin{proof}
If $M$ is maximal Cohen-Macaulay, then the result trivially holds.
If not, then the depth formula holds by Theorem \ref{CMcaseone},
i.e.\
$$\depth M + \depth N = \dim A + \depth (M \otimes_A N).$$
Thus, if the depth of $(M \otimes_A N)$ is nonzero, then so is
$\depth M$.
\end{proof}

\section{Modules with reducible complexity and inifinite complete
intersection dimension}

We shall shortly give examples showing that there exist
modules having reducible complexity of any finite complexity,
but not finite complete
intersection dimension. In order to do this, we opt to work with
complexes in the derived category $D(A)$ of $A$-modules. This is a
triangulated category, the suspension functor $\s$ being the left
shift of a complex together with a sign change in the
differential. Now let
$$C \colon \cdots \to C_{n+1} \to C_n \to C_{n-1} \to \cdots$$
be a complex in $D(A)$. Then $C$ is \emph{bounded below} if
$C_n=0$ for $n \ll 0$, and \emph{bounded above} if $C_n=0$ for $n
\gg 0$. The complex is \emph{bounded} if it is both bounded below
and bounded above. The \emph{homology} of $C$, denoted $\H (C)$,
is the complex with $\H (C)_n = \H_n (C)$, and with trivial
differentials. When $\H (C)$ is bounded and degreewise finitely
generated, then $C$ is said to be \emph{homologically finite}.
We denote the full subcategory of homologically finite complexes
by $D^{hf}(A)$.

When $C$ is homologically finite, it has a minimal free resolution
(cf.\ \cite{Roberts}). Thus, there exists a quasi-isomorphism $F
\simeq C$, where $F$ is a bounded below complex
$$\cdots \to F_{n+1} \xrightarrow{d_{n+1}} F_n \xrightarrow{d_n} F_{n-1} \to \cdots$$
of finitely generated free $A$-modules, and where $\Im d_n \subseteq
\m F_{n-1}$. The minimal free resolution is unique up to
isomorphism, and so for each integer $n$ the rank of the free
module $F_n$ is a well defined invariant of $C$. Thus we may
define Betti numbers and complexity for homologically finite
complexes, and also the concept of reducible complexity.
A complex $C\in D^{hf}(A)$ is said to have finite project dimension
if it is quasi-isomorphic to a perfect complex.

\begin{definition}\label{defredcx}
The full subcategory of complexes in $D^{hf}(A)$ having \emph{reducible complexity} is defined
inductively as follows:
\begin{enumerate}
\item[(i)] Every homologically finite complex of finite projective
dimension has reducible complexity. \item[(ii)] A homologically
finite complex $C$ of finite positive complexity has reducible
complexity if there exists a triangle
$$C \to \s^nC \to K \to \s C$$
with $n>0$, such that $\cx K < \cx C$ and $K$ has reducible
complexity.
\end{enumerate}
\end{definition}

The Betti numbers (and hence also the complexity) of an $A$-module
$M$ equal the Betti numbers of $M$ viewed as an element in $D(A)$,
i.e.\ as the stalk complex
$$\cdots \to 0 \to 0 \to M \to 0 \to 0 \to \cdots$$
with $M$ concentrated in degree zero. Moreover, the module $M$ has
reducible complexity if and only if it has reducible complexity in
$D(A)$. To see this, let
$$\eta : 0 \to M \to K \to \Omega_A^{n-1}(M) \to 0$$
be a short exact sequence, and let $F$ be a free resolution of
$M$. Then $\eta$ corresponds to a map $F \to \s^nF$ in $D(A)$
whose cone is a free resolution of $K$. Thus a sequence of short
exact sequences of modules (with $K_0=M$)
$$
\xymatrixrowsep{.6pc} \xymatrix{ 0 \ar[r] &
K_0 \ar[r] & K_1 \ar[r] & \Omega_A^{n_1-1}(K_0) \ar[r] & 0
\\
& \vdots & \vdots & \vdots \\
0 \ar[r] & K_{t-1} \ar[r] & K_t \ar[r] & \Omega_A^{n_t-1}(K_{t-1})
\ar[r] & 0 }
$$
reducing the complexity of $M$, corresponds to a
sequence of triangles
$$\xymatrixrowsep{.6pc} \xymatrix{
F(K_0) \ar[r] & \s^{n_1}F(K_0) \ar[r] & F(K_1) \ar[r] & \s F(K_0)
\\
\vdots & \vdots & \vdots & \vdots \\
F(K_{t-1}) \ar[r] & \s^{n_t}F(K_{t-1}) \ar[r] & F(K_t) \ar[r] & \s
F(K_{t-1}) }$$ reducing the complexity of $F$, with $F(K_i)$ a free resolution
of $K_i$.
Conversely, every such sequence of triangles of free resolutions of $K_i$
gives a sequence of short exact sequences reducing the complexity of $M$.

There is more generally a relation between homologically finite complexes of
reducible complexity and modules of reducible complexity.
For a complex $C$ in $D^{hf}(A)$ we define the {\em supremum} of $C$
to be
\[
\sup(C)=\sup\{i|\H_i(C)\ne 0\}.
\]

\begin{proposition}  Let $C\in D^{hf}(A)$ be a complex with reducible
complexity and $n=\sup(C)$.  Then the $A$-module $M=\Coker(C_{n+1}\to C_n)$
has reducible complexity.
\end{proposition}

\begin{proof} We may assume that $C$ is a minimal complex of finitely generated
free $A$-modules.  Let $F=C_{\ge n}$. Then $F$ is a minimal free resolution
of $M$.  Moreover, it is easy to check that $F$ has reducible complexity since
$C$ has.  Thus by the discussion above, $M$ has reducible complexity.
\end{proof}

We say that a complex $C\in D^{hf}(A)$ has {\em finite CI-dimension} if there
exists a diagram of local ring homomorphisms $A\to R \leftarrow Q$
with $A\to R$ flat and $R\leftarrow Q$ surjective with kernel
generated by a regular sequence, such that $R\otimes_AC$ has
finite projective dimension as a complex of $Q$-modules (cf. \cite{Sather-Wagstaff}).

There is a connection between finite CI-dimension of a complex and that of
a module.

\begin{proposition} Let $C$ be in $D^{hf}(A)$.
If $\Coker(C_{n+1}\to C_n)$ has finite CI-dimension
for some $n\ge \sup(C)$, then so does $C$.
\end{proposition}

\begin{proof}  We may assume that $C$ is a minimal complex of finitely generated
free $A$-modules. The result is then \cite[Corollary 3.8]{Sather-Wagstaff}.
\end{proof}

The following is an easy fact whose proof is left as an exercise.

\begin{proposition}\label{propses}
Let $0\to Y \to X \to \Sigma^n X \to 0$ be a
short exact sequence of complexes in $D^{hf}(A)$. Then $Y$ has
finite CI dimension if $X$ does.
\end{proposition}

\begin{construction}
Let $k$ be a field and $A^{(i)}$ be $k$-algebras for $1\le i\le c$.  Furthermore,
for each $1\le i\le c$ let $F^{(i)}$ be a complex of finitely generated free $A^{(i)}$-modules
with $F^{(i)}_j=0$ for $j<0$, and possessing  a surjective chain map
$\eta^{(i)}:F^{(i)} \to \Sigma^{n_i} F^{(i)}$ of degree $n_i$. Then
\[
F=F^{(1)}\otimes_k\cdots\otimes_kF^{(c)}
\]
is a complex of finitely generated
free $A=A^{(1)}\otimes_k\cdots\otimes_kA^{(c)}$-modules with $F_j=0$ for $j<0$,
and each $\eta^{(i)}$ induces a surjective chain map $\bar\eta^{(i)}:F\to\Sigma^{n_i}F$.
Moreover the $\bar\eta^{(i)}$ commute with one another.

Let $C(\bar\eta^{(1)})$ denote the cone of $\bar\eta^{(1)}$.  Then since $\eta^{(1)}$ and
$\eta^{(2)}$ commute with one another, $\bar\eta^{(2)}$ induces
a surjective chain map $C(\bar\eta^{(1)}) \to\Sigma^{n_2}C(\bar\eta^{(1)})$.  By abuse of
notation we let $C(\bar\eta^{(2)})$ denote the cone of this chain map.  Inductively
we define $C(\bar\eta^{(i)})$ to be the cone of the surjective chain map
on $C(\bar\eta^{(i-1)})$ induced by $\bar\eta^{(i)}$.

When $\eta^{(i)}_j$ is an isomorphism for $j\ge n_i$, and no such chain map exists of degree
less than $n_i$, we say that $F^{(i)}$ is {\em periodic of period $n_i$}.

\begin{proposition}\label{prop4}
With the notation above, assume that $A$ is local.
Suppose that each $F^{(i)}$ is periodic of period
$n_i$, with $\eta_i: F^{(i)}\to\Sigma^{n_i} F^{(i)}$ being the surjective endomorphism
defining the periodicity of $F^{(i)}$.  Then $F$ has reducible complexity and complexity $c$.
\end{proposition}

\begin{proof} By the discussion above we have a sequence of triangles
\[
\xymatrixrowsep{.6pc} \xymatrix{
F\ar[r] & \s^{n_1}F \ar[r] & C(\bar\eta^{(1)}) \ar[r] & \s F
\\
\vdots & \vdots & \vdots & \vdots \\
C(\bar\eta^{(c-1)}) \ar[r] & \s^{n_c}C(\bar\eta^{(c-1)}) \ar[r] & C(\bar\eta^{(c)}) \ar[r] & \s
C(\bar\eta^{(c-1)}) }
\]
Since each chain map induced by $\bar\eta^{(i)}$, $1\le i\le c$, is onto,
the complexity of each $C(\bar\eta^{(i)})$ is one less than that of $C(\bar\eta^{(i-1)})$.
\end{proof}

Assume that each $F^{(i)}$ is periodic. Define for $0\le i\le c$ the complexes
\[
E^{(i)}= F^{(1)}_{<n_1}\otimes_k\cdots\otimes_k F^{(i)}_{<n_i}
\otimes_k F^{(i+1)}\otimes_k\cdots\otimes_k F^{(c)}.
\]
The chain maps $\eta^{(i)}$ induce short exact sequences
\begin{equation}\label{ses}
0\to E^{(i)}\to E^{(i-1)} \to \Sigma^{n_{i}} E^{(i-1)} \to 0
\end{equation}

\begin{proposition}\label{inf ci}
With the notation above, assume that each $F^{(i)}$ is periodic, $n_i=1$ for
$1\le i\le c-1$ and $n_c>2$.  Then the complex $F$ has infinite CI-dimension.
\end{proposition}

\begin{proof} By applying Proposition \ref{propses} inductively to the short exact sequences
(\ref{ses}),
$E^{(c-1)}$ has finite CI-dimension if $F$ does.  However, if $n_i=1$ for
$1\le i\le c-1$ we have
\[
E^{(c-1)}=F_0^{(1)}\otimes_k\cdots\otimes_k F_0^{(c-1)}\otimes_k F^{(c)}
\]
which is a periodic complex of free $A$-modules of period $n_c>2$.  It is
well-known that complexes of finite CI-dimension and complexity one
are periodic of period $\le 2$.
Thus $E^{(c-1)}$ has infinite CI-dimension, and therefore so does $F$.
\end{proof}

The following corollary is the main point of this section.  Its proof follows from the previous
results.

\begin{corollary}\label{cor}
Assume that $A$ is local, and that $F$ is defined as in Proposition \ref{inf ci}, with $n_i=1$ for $1\le i\le c-1$ and $n_c>2$. Then the $A$-module $M=\Coker(F_1\to F_0)$ has reducible complexity $c$ and infinite CI-dimension.
\end{corollary}

\begin{remark} The hypothesis in \ref{prop4} and \ref{cor}, that $A$ be local, is easy to achieve.  Indeed, one could take each $k$-algebra $A^{(i)}$ to be local and artinian. Then the same holds for $A$.  Examples of complexes $F^{(i)}$ of complexity one, both periodic of arbitrary period, and aperiodic are well-known to exist over local artinian rings.  See, for example \cite{GP}.

In the spirit of Section \ref{secdepthformula}, one would also like to know that there exist modules of reducible complexity $c$ and infinite CI-dimension over rings $A$ of positive dimension.  These are easily seen to exist by taking deformations of examples such as above.  For instance, if $A$ is
local artinian and $M$ is an $A$-module of reducible complexity $c$ and
infinite CI-dimension, then for indeterminates $x_1,\dots,x_r$, let 
$A'$ be the ring $A[x_1,\dots,x_r]$ suitable localized, and $M'$ the $A'$-module $M[x_1,\dots,x_r]$ similarly localized.  Then $M'$ has the
same properties as $M$, now over the positive dimensional local ring $A'$.
This same conclusion holds too if one reduces both $A'$ and $M'$ by a regular sequence in the maximal ideal of $A'$.
\end{remark}

\end{construction}

\section*{Acknowledgements}

This work was done while the second author was visiting Trondheim,
Norway, August 2009.  He thanks the Algebra Group at the Institutt
for Matematiske Fag, NTNU, for their hospitality and generous
support. The first author was supported by NFR Storforsk grant
no.\ 167130.

\end{document}